\documentclass[12pt]{amsart}
\usepackage{amscd,amsmath,amsthm,amssymb}
\usepackage{pstcol,pst-plot,pst-3d}
\usepackage{lmodern,pst-node}
\def\NZQ{\Bbb}               
\def\NN{{\NZQ N}}

\def\FF{{\NZQ F}}
\def\GG{{\NZQ G}}

\def\frk{\frak}               

\def\Phi{{\frk n}}
\def\Phi{{\frk N}}
\def\bb{{\bold b}}
\def\xb{{\bold x}}

\def\opn#1#2{\def#1{\operatorname{#2}}} 
\opn\chara{char} \opn\length{\ell} \opn\pd{pd} \opn\rk{rk}
\opn\projdim{proj\,dim} \opn\injdim{inj\,dim} \opn\rank{rank}
\opn\depth{depth} \opn\grade{grade} \opn\height{height}
\opn\embdim{emb\,dim} \opn\codim{codim}

\opn\Tr{Tr} \opn\bigrank{big\,rank}
\opn\superheight{superheight}\opn\lcm{lcm}
\opn\trdeg{tr\,deg}
\opn\reg{reg} \opn\lreg{lreg} \opn\ini{in} \opn\lpd{lpd}
\opn\size{size}\opn\bigsize{bigsize}
\opn\cosize{cosize}\opn\bigcosize{bigcosize}
\opn\sdepth{sdepth}\opn\sreg{sreg}
\opn\link{link}\opn\fdepth{fdepth}
\opn\div{div} \opn\Div{Div} \opn\cl{cl} \opn\Cl{Cl}
\opn\Spec{Spec} \opn\Supp{Supp} \opn\supp{supp} \opn\Sing{Sing}
\opn\Ass{Ass} \opn\Min{Min}\opn\Mon{Mon} \opn\dstab{dstab} \opn\astab{astab}
\opn\Ann{Ann} \opn\Rad{Rad} \opn\Soc{Soc}
\opn\Im{Im} \opn\Ker{Ker} \opn\Coker{Coker} \opn\Am{Am}
\opn\Hom{Hom} \opn\Tor{Tor} \opn\Ext{Ext} \opn\End{End}
\opn\Aut{Aut} \opn\id{id}

\opn\nat{nat}
\opn\pff{pf}
\opn\Pf{Pf} \opn\GL{GL} \opn\SL{SL} \opn\mod{mod} \opn\ord{ord}
\opn\Gin{Gin} \opn\Hilb{Hilb}\opn\sort{sort}
\opn\aff{aff} \opn\con{conv} \opn\relint{relint} \opn\st{st}
\opn\lk{lk} \opn\cn{cn} \opn\core{core} \opn\vol{vol}
\opn\link{link} \opn\star{star}\opn\lex{lex}
\opn\gr{gr}

\def\pot#1#2{#1[\kern-0.28ex[#2]\kern-0.28ex]}

\opn\dirlim{\underrightarrow{\lim}}
\opn\inivlim{\underleftarrow{\lim}}
\let\Dirsum=\bigoplus

\def\Implies{\ifmmode\Longrightarrow \else
        \unskip${}\Longrightarrow{}$\ignorespaces\fi}
\def\implies{\ifmmode\Rightarrow \else
        \unskip${}\Rightarrow{}$\ignorespaces\fi}
\def\iff{\ifmmode\Longleftrightarrow \else
        \unskip${}\Longleftrightarrow{}$\ignorespaces\fi}

\let\:=\colon
\newtheorem{Theorem}{Theorem}

\newtheorem{Corollary}[Theorem]{Corollary}
\newtheorem{Proposition}[Theorem]{Proposition}
\newtheorem{Remark}[Theorem]{Remark}

\newtheorem{Example}[Theorem]{Example}

\let\epsilon\varepsilon
\let\kappa=\varkappa
\def\qed{\ifhmode\textqed\fi
      \ifmmode\ifinner\quad\qedsymbol\else\dispqed\fi\fi}
\def\textqed{\unskip\nobreak\penalty50
       \hskip2em\hbox{}\nobreak\hfil\qedsymbol
       \parfillskip=0pt \finalhyphendemerits=0}
\def\dispqed{\rlap{\qquad\qedsymbol}}

\opn\dis{dis}
\def\pnt{{\raise0.5mm\hbox{\large\bf.}}}

\opn\Lex{Lex}

\begin{document}

\title  {A note on the subadditivity problem for maximal shifts in free resolutions}
\author{Juergen Herzog and Hema Srinivasan}

\address{J\"urgen Herzog, Fachbereich Mathematik, Universit\"at Duisburg-Essen, Campus Essen, 45117
Essen, Germany}
\email{juergen.herzog@uni-essen.de}

\address{Hema Srinivasan, Mathematics Department, University of
Missouri, Columbia, MO 65211, USA}
\email{SrinivasanH@math.missouri.edu}


 \begin{abstract}
We present some partial results regarding subadditivity of maximal shifts in  finite graded free resolutions.
 \end{abstract}

\thanks{The paper was written while the authors were visiting MSRI at Berkeley. They wish to acknowledge the support, the hospitality and the inspiring atmosphere of this institution.}

\maketitle

Let $K$ be field, $S=K[x_1,\ldots,x_n]$ the polynomial ring over $K$ in the indeterminates $x_1,\ldots,x_n$ and $I\subset S$ a graded ideal. Let $(\FF, \partial)$ be a graded free $S$-resolution of $R=S/I$.   Each free module $\FF_a$ in the resolution  is of the form $\FF_a=\Dirsum_jS(-j)^{b_{aj}}$. We set $$t_a(\FF)= \max\{j\:\; b_{aj}\neq 0\}.$$ In the case that $\FF$ is the graded minimal free resolution of $I$ we  write $t_a(I)$ instead of $t_a(\FF)$.

We say $\FF$ satisfies the {\em subadditivity condition}, if $t_{a+b}(\FF)\leq t_a(\FF)+t_b(\FF)$.

\begin{Remark}{\em  The Taylor complex and the Scarf complex satisfy the subadditivity condition. Indeed, both complexes are cellular resolutions supported on a simplicial complex.  From this fact the  assertion follows immediately.}
\end{Remark}

The minimal resolution of a  graded algebra $S/I$ does not always satisfy the subadditivity condition as pointed out in \cite{ACI}. Additional assumptions on the ideal $I$ are required. Somewhat weaker inequalities can be shown  in certain ranges of $a$ and $b$, and in particular the inequality $t_{a+1}(I)\leq t_a(I)+t_1(I)$ if $R=S/I$ is Koszul and $a\leq \height I$, see \cite[Theorem 4.1]{ACI}. Another case of interest for which the subadditivity condition holds is when $\dim S/I\leq 1$ and $a+b=n$ as shown in \cite[Theorem 4.1]{EHU} by David Eisenbud, Craig Huneke and Bernd Ulrich. No counterexample is known for monomial ideals,

\medskip
For a general graded ideal $I$ we have the following result.

\begin{Proposition}
\label{free}
Let $I\subset S$ be a graded ideal, $\FF$ the graded minimal free resolution of $S/I$. Suppose there exists a homogeneous basis $f_1,\ldots,f_r$ of $F_a$  such  that $$\partial(\FF_{a+1})\subset  \Dirsum_{i=1}^{r-1}Sf_i.$$ Then $\deg f_r\leq t_{a-1}+t_1$.
\end{Proposition}

\begin{proof} We denote by $(\FF^*,\partial^*)$  the complex $\Hom_S(\FF,S)$ which is  dual to $\FF$.  For any basis $h_1,\ldots,h_l$  of $\FF_b$ we denote by $h_i^*$ the basis element of $\FF_b^*$ with $h_i^*(h_j)=1$ if  $j=i$ and  $h_i^*(h_j)=0$, otherwise. Then $h_1^*,\ldots,h_l^*$ is a basis of $\FF_b^*$, the so-called dual basis of $h_1,\ldots,h_l$.

Our assumption implies that $\partial^*(f_r^*)=0$. This implies that $f_r^*$ is a generator of $H^a(\FF^*)=\Ext^a_S(S/I,S)$, and hence $If_r^*=0$  in $H^a(\FF^*)$, since $\Ext^a(S/I,S)$ is an $S/I$-module.  On the other hand, if $g_1,\ldots, g_m$ is a basis of $\FF_{a-1}$ and $\partial(f_r)=c_1g_1+\cdots +c_ng_m$, then $\partial^*(g_i^*)=c_if^*+m_i$ where each $m_i$ is a linear combination of the remaining basis elements of $\FF_a^*$. Let $c\in I$ be a generator of maximal degree. Then by definition, $\deg c=t_1(I)$. Since $If_r^*=0$ in $H^a(\FF^*)$, there exist homogeneous elements $s_i\in S$ such that $cf_r^*=\sum_{i=1}^ms_i(c_if_r^*+m_i)$. This is only possible if $t_1(I)=\deg c_i +\deg s_i$ for some $i$. In particular, $\deg c_i\leq t_1(I)$. It follows that $\deg f_r=\deg c_i+\deg g_i\leq t_1(I)+t_{a-1}(I)$, as desired.
\end{proof}

In \cite[Theorem 4.4]{McC} Jason  McCullough shows that $t_p(I)\leq \max_a\{t_a(I)+t_{p-a}(I)\}$ where $p=\projdim S/I$. As an immediate consequence of Proposition~\ref{free} we obtain the following improvement of  McCullough's inequality:

\begin{Corollary}
\label{good}
Let $I\subset S$ be a graded ideal of projective dimension $p$. Then $$t_p(I)\leq t_{p-1}(I)+t_1(I).$$
\end{Corollary}

For monomial ideals one even has

\begin{Corollary}
\label{monomial}
Let $I$ be a monomial ideal. Then  $t_a(I)\leq t_{a-1}(I)+t_1(I)$ for all $a\geq 1$.
\end{Corollary}

For the proof of this  and the following results we will use the restriction lemma as given in \cite[Lemma 4.4]{HH}: let $I$ be a monomial ideal with multigraded (minimal) free resolution $\FF$ and let $\alpha\in \NN^n$. Then the restricted complex $\FF^{\leq \alpha}$ which is the subcomplex of $\FF$ for which $(\FF^{\leq\alpha})_i$ is spanned by those basis elements of $\FF_i$ whose multidegree is componentwise less than or equal to $\alpha$, is a (minimal) multigraded free resolution
of the monomial ideal $I^{\leq\alpha}$ which is generated by all monomials $\xb^\bb\in I$ with $\bb\leq \alpha$, componentwise.

\begin{proof}[Proof of Corollary~\ref{monomial}] Let $\FF$ the minimal multigraded free $S$-resolution of $S/I$, and let $f\in F_a$ be a homogeneous element of multidegree $\alpha\in \NN^n$ whose total degree is $t_a(I)$. We apply the restriction lemma  and consider the restricted complex $\FF^{\leq \alpha}$.   Let $f_1,\ldots,f_r$ be a homogenous basis of $(\FF^{\leq\alpha})_{a}$ with $f_r=f$. Since there is no basis element of $(\FF^{\leq\alpha})_{a+1}$ of a multidegree which is coefficient bigger than $\alpha$, and since the resolution $\FF^{\leq\alpha}$ is minimal,  it follows that $\partial((\FF^{\leq \alpha})_{a+1})\subset   \Dirsum_{i=1}^{r-1}Sf_i$. Thus we may apply Proposition~\ref{free} and deduce that $t_a(I^{\leq\alpha})\leq t_{a-1}(I^{\leq\alpha})+t_1(I^{\leq\alpha})$. Since $t_a(I)=t_a(I^{\leq\alpha})$,  $t_{a-1}(I^{\leq\alpha})\leq t_{a-1}(I)$  and  $t_{1}(I^{\leq\alpha})\leq t_{1}(I)$, the assertion follows.
\end{proof}

The preceding corollary generalizes \cite[Corollary 1.9]{Gim} of Fern\'andez-Ramos and Philippe Gimenez, who showed that $t_a\leq t_{a-1}+2$ for any monomial ideal generated in degree $2$.

\medskip
Let $I\subset S$ be a monomial ideal, and $\alpha,\beta\in\NN^n$ be two integer vectors. We say that $(\alpha,\beta)$ is a {\em covering pair} for $I$, if
\[
I=I^{\leq \alpha}+I^{\leq \beta}.
\]
\begin{Theorem}
\label{covering}
Let $(\alpha,\beta)$ be a covering pair for the monomial ideal $I$, and suppose that $p=\projdim S/I^{\leq \alpha}$ and $q=\projdim S/I^{\leq \beta}$. Then $\projdim S/I\leq p+q$, and  for all integers $a\leq \projdim S/I$ we have
\[
t_a(I)\leq \max\{t_i(I)+t_j(I)\:\, i+j=a,\; i\leq p,\; j\leq q\}.
\]
\end{Theorem}

\begin{proof}
We consider the complex $\GG=\FF^{\leq \alpha}*\FF^{\leq \beta}$ defined in \cite{He}. Then $\GG$ is a multigraded free resolution of $I^{\leq \alpha}+I^{\leq \beta}$ of length $p+q$, and hence a multigraded free resolution of $I$. In particular, it follows that $\projdim S/I\leq p+q$.

By construction,
\[
\GG_{a}=\Dirsum_{i+j=a}(\FF^{\leq \alpha})_i*(\FF^{\leq \beta})_j,
\]
where each direct summand  $(\FF^{\leq \alpha})_i*(\FF^{\leq \beta})_j$ is a free multigraded $S$-module. If $f_1,\ldots, f_s$ is a multihomogeneous basis of $(\FF^{\leq \alpha})_i$ and $g_1,\ldots,g_r$  a multihomogeneous basis of $(\FF^{\leq \beta})_j$, then the symbols $f_k*g_l$ with  $k=1,\ldots, s$ and $l=1,\ldots,r$ establish a multihomogeneous basis of $(\FF^{\leq \alpha})_i*(\FF^{\leq \beta})_j$, and if $\sigma_k$ is the multidegree of $f_k$ and $\tau_l$ is the multidegree of $g_l$, then $\sigma_k\vee \tau_l$ is the multidegree of $f_k*g_l$, where for two integer vectors  $\gamma,\delta\in \NN^n$ we denote by $\gamma\vee \delta$ the integer vector which is obtained from $\gamma$ and $\delta$ by taking componentwise the maximum. It follows that the element of maximal  (total) degree in $(\FF^{\leq \alpha})_i*(\FF^{\leq \beta})_j$ has degree less than or equal to $t_i(\FF^{\leq \alpha})+t_j(\FF^{\leq \beta})$. Consequently we obtain
\begin{eqnarray*}
t_{a}(I) &=& t_{a}(\FF)\leq t_{a}(\GG)\leq \max\{t_i(\FF^{\leq \alpha})+t_j(\FF^{\leq \beta})\:\; i+j=a \;, i\leq p \;, j\leq q\}\\
&\leq &\max\{t_i(I)+t_j(I)\:\; i+j=a \;, i\leq p \;, j\leq q\}.
\end{eqnarray*}
\end{proof}

The following example illustrates that Theorem~\ref{covering}   leads to inequalities which are not implied  by Corollary~\ref{good}.

\begin{Example}{\em  Let  $S= k[x,y,z,u,v,w,a]$ and $$I = (x^2w^2v^2,a^2x^3y^2u^2w^2,a^2z^2u^2,u^2y^2z^3,x^3y^2z^2,x^5,y^5,z^5,u^5,w^5,
v^6,a^6)\subset S.$$
 We choose  $\alpha = (5,5,5,5,0,0,0)$  and  $\beta = (3,3,2,2,6,5,6) $. Then
$$I^{\le \alpha} = (x^5,y^5,z^5,u^5, x^3y^3z^2,u^2y^2z^3),\;
I^{\le \beta} = (w^5,v^6,a^6,x^2w^2v^2,a^2x^3y^2u^2w^2,a^2z^2u^2).$$
Here,  $p = 4$, $q= 5$ and $\projdim S/I=7$.  Thus by Theorem~\ref{covering}, $$t_7(I) \le \max \{t_2(I)+t_5(I), t_3(I)+t_4(I)\}.$$}
\end{Example}

\begin{Corollary}
\label{range}
Let $s=p+q-a$. Then with the notation and assumptions of Theorem~\ref{covering} we have
\[
t_a(I)\leq\max\{t_i(I)+t_{a-i}(I)\:\; p-s\leq i\leq p\}.
\]
\end{Corollary}

A a special case of this corollary one obtains

\begin{Corollary}
\label{general}
Let $I\subset S=K[x_1,\ldots,x_n]$ be a monomial ideal with $\dim S/I=0$ which is minimally generated by $m\leq 2n-6$ monomials, and let $a$ be an integer with $(m+4)/2\leq a\leq n$. Then for all $p=m-a+2,\ldots,a-2$,
\[
t_a(I)\leq \min\{t_1(I)+t_{a-1}(I), \max\{t_i(I)+t_{a-i}(I)\:\; p-(m-a)\leq i\leq \min\{p,a/2\}\}\}.
\]
\end{Corollary}

\begin{proof}
Due to Corollary~\ref{good} we only need to show that $$t_a(I)\leq\max\{t_i(I)+t_{a-i}(I)\:\; p+a-m\leq i\leq  \min\{p,a/2\}\}\}.$$
Since $\dim S/I=0$, it follows that among the minimal set of  generators $G(I)$ of $I$ are the pure powers $x_1^{a_1},\ldots,x_n^{a_n}$ for suitable $a_i>0$. We let $\alpha=(a_1,\ldots,a_p,0,\ldots,0)$. Then $I^{\leq \alpha}$ has all its generators in $K[x_1,\ldots,x_p]$
so that $\projdim S/I=p$. Let $J$ be the ideal which is generated by the set of  monomials $G(I)\setminus \{x_1^{a_1},\ldots,x_p^{a_p}\}$, and  let $x^\beta$ be the least common multiple of the generators of $J$. Then $J=I^{\leq \beta}$ and $ (\alpha,\beta)$ is a covering pair for $I$. Since $J$ is generated by $m-p$ elements it follows that $q=\projdim S/J\leq m-p$. Hence the desired inequality follows from Corollary~\ref{range}. The conditions on the integers  $a$, $m$ and $p$ only make sure that $i\geq 2$ and $a-i\geq 2$ for all $i$ with  $p+a-m\leq i\leq p$,  and that $m-a+2\leq a-2$.
\end{proof}

The bound in Corollary \ref{general} is a partial  improvement of the results in \cite{EHU} and
\cite{McC}  since  the bound is also valid for certain  $a< n$.  For $a=n$, it is weaker than the one in \cite{EHU} for zero dimensional rings and is stronger than the one in \cite{McC}.  For example, if $n=7$ and $m=8$ one has  $t_6\le t_1+t_2+t_3$, and if $6\le n\le 20$ and $m\le 2n-6$, then one has  $t_7\le t_1+t_2+t_4$.

\begin{Remark}
 \label{multiple}
 {\em With the same  methods as applied in the proof of Theorem~\ref{covering} one can show the following statement: let $I\subset S$ be a monomial ideal with graded minimal free resolution $\FF$, and   $f_i\in F_{a_i}$ multihomogeneous basis elements  of multidegree $\alpha_i$ for $i=1,\ldots,r$.  Assume that $I=\sum_{i=1}^rI^{\leq \alpha_i}$. Then
\[
t_{a_1+a_2+\cdots +a_r}(I)\leq t_{a_1}(I)+t_{a_2}(I)+\cdots +t_{a_r}(I).
\]

}
\end{Remark}

To satisfy the condition $I=\sum_{i=1}^rI^{\leq \alpha_i}$ requires in general that either $r$ is big enough or that the $\alpha_i$ are large enough (with respect to the partial order given by componentwise comparison). Here is an example with $r=2$ to which Remark~\ref{multiple} applies: let $$I =(x^2w^2v^2,a^2x^3y^2u^2w^2,a^2z^2u^2,u^2y^2z^3,x^3y^2z^2)\subset  k[x,y,z,w,u,v,a]$$
The Betti numbers of $R/I$ are $1,5,8,5,1$. Even though the Betti sequence is symmmetric, the   ideal $I$ is not Gorenstein, since it is of height $2$ and projective dimension $4$.
The two multidegrees in $F_2$ which form a covering pair for $I$ are $(3,2,2,2,2,0, 2)$ and $(2,2,3,2,2,2,0)$.  In this example we have
$t_1 = 11, t_2 = 13, t_3 = 15, t_4= 16$ and we clearly have $t_i \le t_2+t_2$ .

\end{document}